\documentclass[11pt]{amsart}

\usepackage[T1]{fontenc}
\usepackage{lmodern}
\usepackage{microtype}
\usepackage{amsmath,amssymb,mathtools}
\usepackage{enumitem}
\usepackage[colorlinks=true,linkcolor=blue,citecolor=blue,urlcolor=blue]{hyperref}
\hypersetup{
  pdftitle={Annular Khovanov homology detects three-strand weaving links},
  pdfauthor={Suman Saurabh}
}

\newtheorem{theorem}{Theorem}[section]
\newtheorem{lemma}[theorem]{Lemma}
\newtheorem{corollary}[theorem]{Corollary}
\theoremstyle{remark}
\newtheorem{remark}[theorem]{Remark}

\newcommand{\F}{\mathbb F_2}
\newcommand{\AKh}{\operatorname{AKh}}
\newcommand{\tr}{\operatorname{tr}}
\newcommand{\e}{\operatorname{e}}

\title[Annular Khovanov homology and weaving links]
{Annular Khovanov homology detects three-strand weaving links}

\author{Suman Saurabh}
\email{realsumansaurabh@gmail.com}


\begin{document}

\begin{abstract}
For \(N\geq1\), let \(K_N\) be the annular closure of
\((\sigma_1\sigma_2^{-1})^N\).  We prove that triply graded annular Khovanov homology
over \(\mathbb F_2\) detects the underlying unoriented annular link \(K_N\).  If
\(3\nmid N\), the only ambiguity is overall orientation reversal.  If \(3\mid N\),
the only ambiguity is independent reversal of components; every such reorientation
has the same homology, so this is sharp.  The proof combines braid detection from the
extremal annular grading with a rigidity theorem: the Jones polynomial and exponent
sum determine \((\sigma_1\sigma_2^{-1})^N\) up to conjugacy in \(B_3\).
\end{abstract}

\maketitle

\section{Introduction}

An invariant \emph{detects} a link when no other link has the same value.  Such results
for Khovanov homology \cite{Khovanov} are rare and often require topology far beyond
the definition of the invariant.  Landmark examples include detection of the unknot
by Kronheimer and Mrowka \cite{KronheimerMrowka}, the trefoils by Baldwin and Sivek
\cite{BaldwinSivek}, the figure-eight knot by Baldwin, Dowlin, Levine, Lidman and
Sazdanovic \cite{BaldwinDowlinLevineLidmanSazdanovic}, and the Hopf links by Baldwin,
Sivek and Xie \cite{BaldwinSivekXie}.

For a link in a thickened annulus, annular Khovanov homology retains information about
the link's position relative to the distinguished axis.  The theory originates in work
of Asaeda, Przytycki and Sikora \cite{APS}; it is related to sutured Floer homology by
Grigsby and Wehrli \cite{GrigsbyWehrli} and carries additional representation-theoretic
structure by work of Grigsby, Licata and Wehrli \cite{GrigsbyLicataWehrli}.  Xie
constructed a spectral sequence to annular instanton Floer homology and obtained
detection results for null-homologous annular links \cite{Xie}.  Most relevant here,
the extremal annular grading distinguishes closed braids from all other annular links
by work of Grigsby and Ni \cite{GrigsbyNi}; the characteristic-two formulation used
below is due to Kim \cite{Kim}.  The relationship between extremal annular gradings and
meridional complexity is studied further by Martin \cite{Martin}.

The extra grading can therefore retain topology that disappears when the braid axis is forgotten. This suggests asking whether it can detect an entire nontrivial family of braid closures, rather than only isolated examples. The family considered below provides such an instance.

Let \(A=S^1\times[0,1]\), and regard the closure of a braid as a link in
\(A\times I\).  Write
\[
  \beta_N=(\sigma_1\sigma_2^{-1})^N\in B_3,
  \qquad K_N=\widehat{\beta_N}\subset A\times I .
\]
We give \(K_N\) its braid orientation unless explicitly stated otherwise.
The underlying link in \(S^3\) is the weaving link \(W(3,N)\), also called a
three-strand Turk's head link.  The permutation of \(\sigma_1\sigma_2^{-1}\) is a
three-cycle, so \(K_N\) is a knot precisely when \(3\nmid N\).  In particular,
\(K_2\) is the figure-eight knot with its standard braid axis.  Mishra and Staffeldt
computed the Jones polynomial and the ranks of ordinary Khovanov homology for the
knots \(W(3,N)\) \cite{MishraStaffeldt}; our concern is instead detection in the
annulus.

We use the triply graded annular Khovanov homology
\(\AKh^{i,j,k}(-;\F)\), where \(i,j,k\) denote the homological, quantum and annular
gradings.  An isomorphism below is understood to preserve all three absolute gradings.

\begin{theorem}\label{thm:main}
Let \(N\geq 1\).  If \(L\subset A\times I\) is an oriented annular link satisfying
\[
  \AKh(L;\F)\cong \AKh(K_N;\F),
\]
then the underlying unoriented link of \(L\) is ambiently isotopic in \(A\times I\)
to \(K_N\).  More precisely:
\begin{enumerate}[label=\textup{(\roman*)}]
\item if \(3\nmid N\), then \(L\) is oriented isotopic to \(K_N\) or to its overall
orientation reverse;
\item if \(3\mid N\), then \(L\) is oriented isotopic to a link obtained from \(K_N\)
by independently reversing some of its three components.  Conversely, every such
reorientation has annular Khovanov homology isomorphic to \(\AKh(K_N;\F)\), preserving
all three absolute gradings.
\end{enumerate}
\end{theorem}

The classical rigidity statement that drives the proof may be of independent use.
For \(\alpha\in B_3\), let \(\e(\alpha)\) denote its exponent sum, and let
\(V_{\widehat\alpha}(t)\) be the Jones polynomial of its braid-oriented closure,
normalized by \(V_{\bigcirc}(t)=1\).

\begin{theorem}\label{thm:rigidity}
Let \(N\geq 1\) and \(\alpha\in B_3\).  If
\[
  \e(\alpha)=0
  \qquad\text{and}\qquad
  V_{\widehat\alpha}(t)=V_{K_N}(t),
\]
then \(\alpha\) is conjugate to \(\beta_N\) in \(B_3\).
\end{theorem}

Theorem~\ref{thm:rigidity} is deliberately phrased in terms of the ordinary Jones
polynomial and one integer.  Both are already contained in the annular Jones
polynomial.  Consequently, the classical part of the proof holds for every \(N\).

\begin{corollary}\label{cor:annular-jones}
For every \(N\geq 1\), the annular Jones polynomial detects \(K_N\) among
braid-oriented braid closures in \(A\times I\).
\end{corollary}

Recent annular detection results include Binns's theorem that integral annular
Khovanov homology detects \(\widehat{\sigma_1\sigma_2^n}\) for
\(-2\leq n\leq 5\), as well as the Mazur pattern \cite{Binns}.  Theorem
\ref{thm:main} supplies a uniform infinite family with alternating signs.  Its proof
has two independent parts.  The extremal annular class first forces a competing link
to be a closed three-braid and recovers its exponent sum, even when the given component
orientations are mixed.  The Jones polynomial then identifies the braid through
Murasugi normal form and adjacent coefficients of a Burau trace.  In particular, no
computation of the full annular Khovanov homology of \(K_N\) is required.

The orientation qualification in Theorem~\ref{thm:main}(ii) is unavoidable.  For
\(N=3m\), the three pairwise linking numbers of \(K_N\) all vanish.  Consequently,
reversing any sublink preserves the absolute gradings on annular Khovanov homology.
On the other hand, reversing one component changes the total winding number in the
annulus from \(3\) to \(1\), so the resulting oriented link is isotopic neither to
\(K_N\) nor to its overall reverse.  The point of the proof is that these obvious
reorientations are the only ambiguity.

\section{Annular and three-braid preliminaries}

\subsection{The extremal annular class}

We use the normalization in which the braidlike resolution of the closure of an
\(n\)-braid \(\alpha\), with all essential circles labelled by \(v_+\), has tridegree
\[
  (i,j,k)=(0,\e(\alpha)+n,n).
\]
The following facts are standard consequences of the definition and of the braid
detection theorem of Grigsby and Ni \cite[Corollary~1.2]{GrigsbyNi}.  The
formulation over \(\mathbb F_2\), in terms of the largest nonzero annular grading, is
\cite[Proposition~2.2.8]{Kim}; see also the use of braid detection in
\cite[proof of Theorem~3.11]{Binns}.

\begin{lemma}\label{lem:extremal}
Let \(\alpha\in B_n\), and give \(\widehat\alpha\) the braid orientation.  Then
\[
  \AKh^{0,\e(\alpha)+n,n}(\widehat\alpha;\F)\cong\F,
\]
all other summands in annular grading \(n\) vanish, and all summands in annular
grading greater than \(n\) vanish.

Conversely, if the total rank of \(\AKh(L;\F)\) in its largest nonzero annular
grading is one, then \(L\) is isotopic to a closed braid.  Its braid index is that
largest annular grading.
\end{lemma}

\begin{proof}
For a braid closure, cut along a meridional disk.  All complete resolutions of the
resulting tangle except the braidlike one backtrack.  The top annular summand is
therefore one-dimensional.  In the usual normalized Khovanov cube, the braidlike
resolution uses one \(1\)-resolution for every negative crossing, so its homological
degree is zero.  Labelling its \(n\) essential circles by \(v_+\) gives quantum degree
\(n+\e(\alpha)\) and annular degree \(n\).  The converse is the braid-detection
theorem cited above.
\end{proof}

Define the annular Jones polynomial by
\[
  \mathcal J_A(L;q,z)
  =\sum_{i,j,k}(-1)^i q^jz^k
       \dim_{\F}\AKh^{i,j,k}(L;\F).
\]
Its specialization at \(z=1\) is the usual Khovanov graded Euler characteristic and
hence determines the ordinary Jones polynomial.  Lemma~\ref{lem:extremal} gives
\begin{equation}\label{eq:top-annular-term}
  [z^n]\mathcal J_A(\widehat\alpha;q,z)=q^{\e(\alpha)+n}
\end{equation}
for a braid-oriented \(n\)-braid closure.

\subsection{Reorienting components}

Let \(L\) be an oriented link and let \(S\) be a union of its components.  Denote by
\(L^S\) the oriented link obtained by reversing every component in \(S\), and set
\[
  \lambda_L(S)
  =\sum_{\substack{C\in\pi_0(S)\\ C'\in\pi_0(L\setminus S)}}
     \operatorname{lk}_L(C,C').
\]
Here and below, the linking numbers on the right are computed before the reversal.
We use the standard inclusion of the thickened annulus into \(S^3\) when forming
these linking numbers.

\begin{lemma}\label{lem:reorientation}
There is a canonical identification of the unnormalized annular Khovanov cubes of
\(L\) and \(L^S\).  Under this identification, a class of tridegree \((i,j,k)\) for
\(L\) has tridegree
\[
  (i-2\lambda_L(S),\ j-6\lambda_L(S),\ k)
\]
for \(L^S\).  Equivalently,
\[
  \AKh^{i,j,k}(L^S;\F)
  \cong
  \AKh^{i+2\lambda_L(S),\,j+6\lambda_L(S),\,k}(L;\F).
\]
Moreover,
\begin{equation}\label{eq:jones-reorientation}
  V_{L^S}(t)=t^{-3\lambda_L(S)}V_L(t).
\end{equation}
\end{lemma}

\begin{proof}
Fix a diagram of \(L\), and let \(p\) and \(r\) be the numbers of positive and
negative crossings, respectively, between \(S\) and its complement.  Then
\(p-r=2\lambda_L(S)\).  Reversing \(S\) interchanges positive and negative signs at
exactly these crossings.  Hence
\[
 n_-(L^S)-n_-(L)=2\lambda_L(S),
 \qquad
 n_+(L^S)-n_+(L)=-2\lambda_L(S).
\]
The unoriented cube, its differential over \(\F\), and the annular grading are
unchanged.  In particular, reorientation changes no cube vertex or smoothing; only
the normalizing shifts depend on the crossing signs.  For a labelled resolution \(x\)
at a vertex \(v\), the normalized
gradings are
\[
 i(x)=|v|-n_-,
 \qquad
 j(x)=\deg(x)+|v|+n_+-2n_-.
\]
The displayed changes in \(n_+\) and \(n_-\) therefore give the asserted shifts.

For completeness, the Jones formula follows directly from the Kauffman bracket.
The writhe changes by \(-4\lambda_L(S)\), while the unoriented bracket is unchanged.
With \(V_L(t)=(-A)^{-3w(D)}\langle D\rangle\) and \(t=A^{-4}\), the ratio is
\(A^{12\lambda_L(S)}=t^{-3\lambda_L(S)}\).
\end{proof}

\begin{lemma}\label{lem:target-linking}
If \(N=3m\), then all three pairwise linking numbers of the braid-oriented link
\(K_N\) vanish.
\end{lemma}

\begin{proof}
Label the strands at the top of a block \((ab^{-1})^3\) by \(1,2,3\), in order.
Tracking labels through that block gives
\[
  123\xrightarrow{a}213\xrightarrow{b^{-1}}231
  \xrightarrow{a}321\xrightarrow{b^{-1}}312
  \xrightarrow{a}132\xrightarrow{b^{-1}}123.
\]
Thus the six crossings, listed in order, have
\[
\begin{array}{c|cccccc}
\text{pair of labels}&12&13&23&12&13&23\\
\text{crossing sign}&+&-&+&-&+&-
\end{array}
\]
Each pair consequently has signed crossing sum zero.  Since the block is pure and
\(\beta_{3m}=((ab^{-1})^3)^m\), repetition of the block shows that every pairwise
linking number, one half of the corresponding signed crossing sum, is zero.
\end{proof}

\begin{corollary}\label{cor:target-reorientations}
If \(3\mid N\), then for every union \(S\) of components of \(K_N\),
\[
  \AKh(K_N^S;\F)\cong\AKh(K_N;\F)
\]
as absolutely triply graded vector spaces.
\end{corollary}

\begin{proof}
Lemma~\ref{lem:target-linking} gives \(\lambda_{K_N}(S)=0\) for every \(S\), so the
claim is immediate from Lemma~\ref{lem:reorientation}.
\end{proof}

\subsection{The Burau representation and Murasugi normal form}

Put \(a=\sigma_1\), \(b=\sigma_2\), and \(\Delta=aba\), so that
\(\Delta^2=(ab)^3\) generates the center of \(B_3\).  We use the reduced Burau
representation
\begin{equation}\label{eq:burau}
  \psi_t(a)=
  \begin{pmatrix}-t&1\\0&1\end{pmatrix},
  \qquad
  \psi_t(b)=
  \begin{pmatrix}1&0\\t&-t\end{pmatrix}.
\end{equation}
It satisfies
\begin{equation}\label{eq:center}
  \psi_t(\Delta^2)=t^3 I_2.
\end{equation}
Birman's formula for the Jones polynomial of a closed three-braid is
\begin{equation}\label{eq:birman}
  V_{\widehat\alpha}(t)
  =(-\sqrt t)^{\e(\alpha)}
    \bigl(t+t^{-1}+\tr\psi_t(\alpha)\bigr)
\end{equation}
\cite{Birman}.  In particular, if two three-braids have the same exponent sum zero,
then equality of their Jones polynomials is equivalent to equality of their Burau
traces.

We will use Murasugi's conjugacy normal form \cite[Proposition~2.1]{Murasugi}; a
convenient modern statement appears in \cite[Definition~4.15]{Truol}.  Every
three-braid is conjugate to a braid of one of the following types:
\begin{align}
  &\Delta^{2\ell}a^p \quad\text{or}\quad \Delta^{2\ell+1},
       &&\ell,p\in\mathbb Z; \label{eq:nf-a}\\
  &\Delta^{2\ell}ab \quad\text{or}\quad \Delta^{2\ell}(ab)^2,
       &&\ell\in\mathbb Z; \label{eq:nf-b}\\
  &\Delta^{2\ell}a^{-p_1}b^{q_1}\cdots a^{-p_r}b^{q_r},
       &&\ell\in\mathbb Z,\quad r\geq1,\quad p_i,q_i\geq1. \label{eq:nf-c}
\end{align}
The alternatives, with the standard cyclic ambiguity in the last line, classify
conjugacy classes.

Conjugation by \(\Delta\) interchanges \(a\) and \(b\).  Hence
\(\Delta(ab^{-1})\Delta^{-1}=ba^{-1}\), and conjugating the latter by \(b^{-1}\)
gives \(a^{-1}b\).  Therefore
\begin{equation}\label{eq:target-conjugate}
  \beta_N=(ab^{-1})^N
  \quad\text{is conjugate to}\quad
  \gamma_N=(a^{-1}b)^N.
\end{equation}
The latter is already in the form \eqref{eq:nf-c}.

\section{Extremal coefficients of a Burau trace}

For a Laurent polynomial \(f\), write \([s^d]f\) for the coefficient of \(s^d\).
For positive integers \(p,q\), set
\[
 X_p=
 \begin{pmatrix}
 s^p&1+s+\cdots+s^{p-1}\\0&1
 \end{pmatrix},
 \qquad
 Y_q=
 \begin{pmatrix}
 1&0\\1+s^{-1}+\cdots+s^{-(q-1)}&s^{-q}
 \end{pmatrix}.
\]

\begin{lemma}\label{lem:trace-coefficients}
Let
\[
  T(s)=\tr\bigl(X_{p_1}Y_{q_1}\cdots X_{p_r}Y_{q_r}\bigr),
  \qquad
  P=\sum_{i=1}^r p_i,
  \quad Q=\sum_{i=1}^r q_i.
\]
Then the largest and smallest exponents occurring in \(T\) are \(P\) and \(-Q\),
respectively, and
\begin{align*}
  [s^P]T&=1,& [s^{P-1}]T&=r,\\
  [s^{-Q}]T&=1,& [s^{-Q+1}]T&=r.
\end{align*}
\end{lemma}

\begin{proof}
All matrix entries have nonnegative coefficients, so there can be no cancellation.
Expand the trace as a sum of weights of closed paths through the two states
\(\{1,2\}\), one transition for each matrix factor.  The constant path at state
\(1\) has weight \(s^P\), while the constant path at state \(2\) has weight
\(s^{-Q}\).

Compare a nonconstant path with the constant path at state \(1\).  Every excursion
from state \(1\) to state \(2\) begins in some \(X_{p_i}\) and ends in a later
\(Y_{q_j}\), with indices understood cyclically.  Such an excursion lowers the
degree by at least one.  The loss is exactly one only when the path leaves in
\(X_{p_i}\), immediately returns in \(Y_{q_i}\), and uses the monomials
\(s^{p_i-1}\) and \(1\) in the two off-diagonal entries.  There are exactly \(r\)
such paths, each with coefficient one.  This proves the two assertions at the top
degree.

The bottom degree is analogous.  Relative to the constant path at state \(2\), a
nonconstant path raises the degree by at least one.  Equality occurs precisely when
it leaves state \(2\) in \(Y_{q_i}\), using \(s^{-(q_i-1)}\), and returns at the next
\(X\)-factor, using its constant off-diagonal monomial.  Again there are \(r\)
choices.  When \(P=Q=r=1\), the two adjacent exponents coincide; both descriptions
refer to the same unique path, and the stated coefficient is still one.
\end{proof}

The matrices in Lemma~\ref{lem:trace-coefficients} are exactly a conjugate of the
Burau matrices suited to Murasugi normal form.  Indeed, for
\(D_s=\operatorname{diag}(-s,1)\), define
\[
  \rho_s(\alpha)=D_s^{-1}\psi_{-s^{-1}}(\alpha)D_s.
\]
A direct calculation gives
\begin{equation}\label{eq:rho}
  \rho_s(a^{-p})=X_p,
  \qquad
  \rho_s(b^q)=Y_q,
  \qquad
  \rho_s(\Delta^{2\ell})=(-1)^\ell s^{-3\ell}I_2.
\end{equation}

\begin{proof}[Proof of Theorem~\ref{thm:rigidity}]
By \eqref{eq:target-conjugate}, it is enough to prove that \(\alpha\) is conjugate to
\(\gamma_N=(a^{-1}b)^N\).  Since both exponent sums are zero, Birman's formula
\eqref{eq:birman} gives
\begin{equation}\label{eq:trace-equality}
  \tr\psi_t(\alpha)=\tr\psi_t(\gamma_N).
\end{equation}

We first rule out the exceptional normal forms \eqref{eq:nf-a} and
\eqref{eq:nf-b}.  At \(t=-1\), the matrices \(\psi_{-1}(a)\) and
\(\psi_{-1}(b)\) are parabolic and \(\psi_{-1}(\Delta^2)=-I_2\).  Thus the
absolute value of the trace is \(2\) for \(\Delta^{2\ell}a^p\), zero for
\(\Delta^{2\ell+1}\), and one for either braid in \eqref{eq:nf-b}.  On the other
hand,
\[
  \psi_{-1}(a^{-1}b)
  =\begin{pmatrix}2&-1\\-1&1\end{pmatrix}
\]
has determinant one and trace three.  If \(u_N\) is the trace of its \(N\)-th
power, then \(u_0=2\), \(u_1=3\), and
\(u_N=3u_{N-1}-u_{N-2}\).  Hence \(u_N>2\) for every \(N\geq1\), contradicting
\eqref{eq:trace-equality} for the exceptional forms.

Consequently, \(\alpha\) is conjugate to
\[
  \Delta^{2\ell}a^{-p_1}b^{q_1}\cdots a^{-p_r}b^{q_r},
  \qquad p_i,q_i\geq1.
\]
Put \(P=\sum p_i\) and \(Q=\sum q_i\).  After substituting \(t=-s^{-1}\),
equations \eqref{eq:trace-equality} and \eqref{eq:rho} identify
\[
  (-1)^\ell s^{-3\ell}
  \tr(X_{p_1}Y_{q_1}\cdots X_{p_r}Y_{q_r})
  \quad\text{with}\quad
  \tr((X_1Y_1)^N).
\]
Lemma~\ref{lem:trace-coefficients} compares the extremal exponents and the
coefficients adjacent to them.  It yields
\begin{equation}\label{eq:degree-equations}
  P-3\ell=N,
  \qquad -Q-3\ell=-N,
  \qquad r=N.
\end{equation}
The last equality follows by taking absolute values of the adjacent coefficients;
the central factor contributes only the sign \((-1)^\ell\).  The first two equations
give \(P+Q=2N\).  Since every \(p_i,q_i\) is positive and \(r=N\), we have
\(P\geq N\) and \(Q\geq N\).  Equality of their sum therefore forces
\[
  P=Q=N
  \quad\text{and}\quad
  p_i=q_i=1\quad(1\leq i\leq N).
\]
The first equation in \eqref{eq:degree-equations} now gives \(\ell=0\).  Thus
\(\alpha\) is conjugate to \((a^{-1}b)^N=\gamma_N\), as required.
\end{proof}

The coefficient mechanism above is related to the analysis of Jones polynomials of
closed alternating three-braids in \cite{Chbili}.  The path proof is included to fix
all signs, central shifts and the case in which the two adjacent exponents coincide.

\section{Detection}

\begin{proof}[Proof of Theorem~\ref{thm:main}]
Let \(L\) satisfy the stated trigraded isomorphism.  Taking the graded Euler
characteristic and forgetting the annular variable gives
\begin{equation}\label{eq:jones-main}
  V_L(t)=V_{K_N}(t).
\end{equation}
By Lemma~\ref{lem:extremal}, \(\AKh(K_N;\F)\) has rank one in its largest nonzero
annular grading, namely \(k=3\).  The same is true of \(L\), so the converse part of
that lemma shows that the underlying unoriented link of \(L\) is the closure of a
three-braid \(\alpha\).  Give \(\widehat\alpha\) its braid orientation, denoted \(O\).
The given orientation on \(L\) is \(O^S\) for some union \(S\) of components.  Put
\(\lambda=\lambda_{(\widehat\alpha,O)}(S)\).

In the braid orientation, the unique class in annular grading three has tridegree
\[
  (0,\e(\alpha)+3,3)
\]
by Lemma~\ref{lem:extremal}.  For \(K_N\), the corresponding tridegree is
\((0,3,3)\), because \(\e(\beta_N)=0\).  Lemma~\ref{lem:reorientation} places the
unique top-annular class for \(O^S\) in tridegree
\[
  (-2\lambda,\ \e(\alpha)+3-6\lambda,\ 3).
\]
Equality with the target's tridegree first gives \(\lambda=0\), from the homological
grading, and then \(\e(\alpha)=0\), from the quantum grading.

Equation~\eqref{eq:jones-reorientation} and \(\lambda=0\) show that the Jones
polynomial in the braid orientation is still the polynomial in
\eqref{eq:jones-main}.  Theorem~\ref{thm:rigidity} therefore implies that \(\alpha\)
is conjugate to \(\beta_N\).  Conjugate braids have ambiently isotopic annular
closures, carrying braid orientation to braid orientation.  It follows that \(L\)
is obtained from \(K_N\) by reversing a union of components.

If \(3\nmid N\), then \(K_N\) is a knot, so this is either no reversal or overall
reversal.  If \(3\mid N\), it is an arbitrary component reorientation.  The converse
assertion in this case is Corollary~\ref{cor:target-reorientations}.
\end{proof}

\begin{proof}[Proof of Corollary~\ref{cor:annular-jones}]
Suppose that a braid-oriented closure \(\widehat\alpha\) has the same annular Jones
polynomial as \(K_N\).  Formula \eqref{eq:top-annular-term} shows first that
\(\alpha\) has three strands and then that \(\e(\alpha)=0\).  Specializing the
annular variable to one gives
\(V_{\widehat\alpha}(t)=V_{K_N}(t)\).  Theorem~\ref{thm:rigidity} now implies that
\(\alpha\) is conjugate to \(\beta_N\), so their annular closures are isotopic.
\end{proof}

\begin{remark}\label{rem:ordinary}
Theorem~\ref{thm:main} is an annular detection theorem.  It does not imply that
ordinary Khovanov homology detects the underlying weaving link \(W(3,N)\).  The
annular grading is used twice: its extremal rank forces a braid closure, and its
extremal tridegree recovers the reorientation linking sum and the exponent sum.
\end{remark}


\begin{thebibliography}{99}

\bibitem{APS}
M.~M. Asaeda, J.~H. Przytycki and A.~S. Sikora,
\emph{Categorification of the Kauffman bracket skein module of \(I\)-bundles over surfaces},
Algebr. Geom. Topol. \textbf{4} (2004), 1177--1210.

\bibitem{BaldwinDowlinLevineLidmanSazdanovic}
J.~A. Baldwin, N. Dowlin, A.~S. Levine, T. Lidman and R. Sazdanovic,
\emph{Khovanov homology detects the figure-eight knot},
Bull. Lond. Math. Soc. \textbf{53} (2021), no.~3, 871--876.

\bibitem{BaldwinSivek}
J.~A. Baldwin and S. Sivek,
\emph{Khovanov homology detects the trefoils},
Duke Math. J. \textbf{171} (2022), no.~4, 885--956.

\bibitem{BaldwinSivekXie}
J.~A. Baldwin, S. Sivek and Y. Xie,
\emph{Khovanov homology detects the Hopf links},
Math. Res. Lett. \textbf{26} (2019), no.~5, 1281--1290.

\bibitem{Binns}
F. Binns,
\emph{Closures of \(3\)-braids and detection},
Pacific J. Math. \textbf{340} (2026), 1--36.

\bibitem{Birman}
J.~S. Birman,
\emph{On the Jones polynomial of closed \(3\)-braids},
Invent. Math. \textbf{81} (1985), no.~2, 287--294.

\bibitem{Chbili}
N. Chbili,
\emph{A note on the Jones polynomials of \(3\)-braid links},
Siberian Math. J. \textbf{63} (2022), no.~5, 983--994.

\bibitem{GrigsbyLicataWehrli}
J.~E. Grigsby, A.~M. Licata and S.~M. Wehrli,
\emph{Annular Khovanov homology and knotted Schur--Weyl representations},
Compos. Math. \textbf{154} (2018), no.~3, 459--502.

\bibitem{GrigsbyNi}
J.~E. Grigsby and Y. Ni,
\emph{Sutured Khovanov homology distinguishes braids from other tangles},
Math. Res. Lett. \textbf{21} (2014), no.~6, 1263--1275.

\bibitem{GrigsbyWehrli}
J.~E. Grigsby and S.~M. Wehrli,
\emph{Khovanov homology, sutured Floer homology and annular links},
Algebr. Geom. Topol. \textbf{10} (2010), no.~4, 2009--2039.

\bibitem{Khovanov}
M. Khovanov,
\emph{A categorification of the Jones polynomial},
Duke Math. J. \textbf{101} (2000), no.~3, 359--426.

\bibitem{Kim}
J. Kim,
\emph{Annular links with \(\mathfrak{sl}_2\)-irreducible annular Khovanov homology},
Ph.D. thesis, California Institute of Technology, 2021,
{doi:10.7907/rwqc-q126}.

\bibitem{KronheimerMrowka}
P.~B. Kronheimer and T.~S. Mrowka,
\emph{Khovanov homology is an unknot-detector},
Publ. Math. Inst. Hautes \`Etudes Sci. \textbf{113} (2011), 97--208.

\bibitem{Martin}
G. Martin,
\emph{Annular Khovanov homology and meridional disks},
J. Knot Theory Ramifications \textbf{32} (2023), no.~2, Paper No.~2250088.

\bibitem{MishraStaffeldt}
R. Mishra and R. Staffeldt,
\emph{Polynomial invariants, knot homologies, and higher twist numbers
of weaving knots \(W(3,n)\)},
J. Knot Theory Ramifications \textbf{30} (2021), no.~4,
Paper No.~2150025, 57 pp.

\bibitem{Murasugi}
K. Murasugi,
\emph{On closed \(3\)-braids},
Mem. Amer. Math. Soc. \textbf{151} (1974).

\bibitem{Truol}
P. Tru\"ol,
\emph{The upsilon invariant at \(1\) of \(3\)-braid knots},
Algebr. Geom. Topol. \textbf{23} (2023), 3763--3804.

\bibitem{Xie}
Y. Xie,
\emph{Instantons and annular Khovanov homology},
Adv. Math. \textbf{388} (2021), Paper No.~107864.

\end{thebibliography}
\end{document}